\documentclass{amsart}

\usepackage{graphicx,amssymb,latexsym,amsfonts,txfonts,amsmath,amsthm}
\usepackage{pdfsync,enumitem}
\usepackage[all,cmtip]{xy}
\input epsf

\usepackage{hyperref}
\hypersetup{
    colorlinks=true,       
    linkcolor=blue,          
    citecolor=blue,        
    filecolor=blue,      
    urlcolor=blue           
}

\newtheorem{theo}{Theorem}

\newtheorem{lemm}[theo]{Lemma}

\theoremstyle{remark}
\newtheorem{rema}[theo]{\bf Remark}

\begin{document} 
\title{Description of Origamis by Schottky groups} 
\date{}

\author{Rub\'en A. Hidalgo} 
\thanks{Partially supported by Project Fondecyt 1230001} 
\address{Departamento de Matem\'atica y Estad\'{\i}stica, Universidad de La Frontera, Temuco, Chile} 
\email{ruben.hidalgo@ufrontera.cl}

\begin{abstract}
Let $(S,\eta)$ be an origami pair, that is, $S$ is a closed Riemann surface of genus $g \geq1$ and $\eta:S \to E$ is a holomorphic branched covering, with at most one branch value, where $E$ is a genus one Riemann surface. As the lowest uniformizations of $S$ are provided by Schottky groups, we are interested in describing origami pairs in terms of virtual Schottky groups. In other words, we are interested in those Kleinian groups $K$ which contain, as a finite index subgroup, a Schottky group $\Gamma$ such that 
$S=\Omega/\Gamma$ and such that $\eta$ is induced by the inclusion $\Gamma \leq K$. We say that $K$ is an origami-Schottky group.
We provide a geometrical structural picture, in terms of the Klein-Maskit combination theorems, of these origami-Schottky groups. 
\end{abstract}

\maketitle

\section{Introduction}
An origami is a combinatorial way of assembling squares (called the faces) to obtain a closed topological (connected, second countable, and orientable) surface $X$ of some genus $g \geq 1$. These objects were introduced by W. Thurston in his classification of diffeomorphisms of surfaces  \cite{Thurston}  and later considered by W. Veech in \cite{Veech}. Some references on generalities on origamis are given by  \cite{Lochak,Gabi1,Gabi2}. Let $T^{2}=S^{1} \times S^{1}$ (a surface of genus one). 
An origami might be thought of a pair $(X,\eta)$, where $\eta:X \to T^{2}$ is a topological branched cover branched in at most one point (with a branch value if and only if $g \geq 2$) and whose degree is equal to the number of faces of the origami (and conversely, each topologically branched cover as above defines an origami). 
If we fix a Riemann surface structure $E$ on the torus $T^{2}$ and we lif it under $\eta$, then we obtain a Riemann surface structure $S$ on $X$ making $\eta:S \to E$ a  holomorphic branched covering with at most one branch value. In this case, we say that $S$ is an {\it origami curve}, that $\eta$ is an {\it origami map}, and that $(S,\eta)$ an {\it origami pair}. This point of view permits to obtain a holomorphic embedded curve in the moduli space ${\mathcal M}_{g}$ of genus $g$ Riemann surfaces which are geodesic for the associated Teichm\"uller metric and that can be defined by algebraic equations with coefficients in $\overline{\mathbb Q}$ (i.e., a Belyi curve, equivalently, a Grothendieck's dessin d'enfant)).
If $\eta$ is a regular branched covering, then the origami pair is called {\it regular}. In this case, if $G<{\rm Aut}(S)$ is the deck group of $\eta$ and $S$ has genus $g \geq 2$, then 
by the Riemann-Hurwitz formula  $|G|\leq 4(g-1)$ \cite{SW}; if $|G|=4(g-1)$, then the origami curve is also known as a {\it Hurwitz translation surface}.

Let $(S,\eta)$ be an origami pair of genus $g \geq 2$,  $p \in E$ be the branched value of $\eta:S \to E$, and $n \geq 2$ be the branching order of $p$. Let us denote by ${\mathcal O}$ the Riemann orbifold whose underlying Riemann surface is $E$ and has a unique cone point at $p$ or cone order $n$.

The highest uniformizations of $S$ are provided by the co-compact torsion-free Fuchsian groups of genus $g$, and the lowest ones are provided by the Schottky groups of rank $g$ (i.e., a purely loxodromic Kleinian group with a non-empty region of discontinuity, and isomorphic to the free group of rank $g$) \cite{Maskit:Schottky groups, Maskit:book}.

The description of the above origami pair, in terms of Fuchsian groups, is well known. By the uniformization theorem, there is a 
a Fuchsian group $F_{n}=\langle A,B: [A,B]^{n}=1\rangle$ (where $A$ and $B$ are hyperbolic elements and $[A,B]=ABA^{-1}B^{-1}$) acting on the hyperbolic plane ${\mathbb H}^{2}$ such that ${\mathbb H}^{2}/F_{n}={\mathcal O}$. The map $\eta$ is induced by a finite index subgroup $\Gamma$ of $F_{n}$ (of index equal to the degree of $\eta$) such that the Riemann orbifold ${\mathbb H}^{2}/\Gamma$ has underlying Riemann surface structure $S$ ($\Gamma$ might have torsion). So, in terms of Fuchsian groups, origami pairs correspond to pairs $(F_{n},\Gamma)$, where $n \geq 2$ and $\Gamma$ is a finite index subgroup of $F_{n}$ (the origami pair is regular if and only if $\Gamma$ is a finite index torsion-free normal subgroup of $F_{n}$). 

In this paper, we are interested in describing origamis in terms of Schottky groups. For this, we need to consider those Kleinian groups $K$, with a non-empty region of discontinuity $\Omega$, such that: (i) the quotient Riemann orbifold ${\mathcal O}_{K}=\Omega/K$ has genus one and has exactly one cone point of order $n$, and (ii) $K$ contains a Schottky group $\Gamma$ as a finite index subgroup. We say that $K$ is an {\it origami-Schottky group of type $n \geq 2$}. Note that
as $\Gamma$ has a finite index in $K$, both have the same region of discontinuity and, in this case, $S_{\Gamma}=\Omega/\Gamma$ is an origami curve,  the inclusion $\Gamma \leq K$ induces 
an origami map $\eta_{K,\Gamma}:\Omega/\Gamma \to \Omega/K$, and $(S_{\Gamma},\eta_{K,\Gamma})$ is an origami pair.

The origami pair $(S,\eta)$ is said to be of {\it Schottky type} if there is an origami-Schottky group $K$, containing a finite index Schottky subgroup $\Gamma$, such that $(S,\eta)$ is equivalent to $(S_{\Gamma},\eta_{K,\Gamma})$, that is, if there are biholomorphisms $\varphi:S \to S_{K}$ and $\psi:E \to E_{K}$ such that $\eta_{K.\Gamma} \circ \varphi=\psi \circ \eta$. We should note that there are origami pairs that are not of Schottky type (see Remark \ref{obs3}).

 As an origami-Schottky group $K$ is a particular example of a function group, the general classification, due to Maskit  \cite{Maskit:function2, Maskit:function3, Maskit:function4}, asserts that $K$ can be constructed, in terms of the Klein-Maskit combination theorems  \cite{Maskit:Comb, Maskit:Comb4, Maskit:book}, by using of the so-called basic groups. As $K$ is geometrically finite and does not have parabolic transformations (since the Schottky groups $\Gamma$ is purely loxodromic and of a finite index), in the construction, there are only used finite groups and/or cyclic loxodromic groups. This paper aims to make this structural description for origami-Schottky groups in a more explicit form.

\begin{theo}[Structural description of origami-Schottky groups]\label{const0}
Let $n \geq 2$ be an integer.
\begin{enumerate}[leftmargin=*,align=left]
\item If $K$ is an origami-Schottky group of type $n$, then (up to conjugation by a suitable M\"obius transformation) one of the following holds. (see Figure \ref{figura1})

\begin{enumerate}[leftmargin=*,align=left]
\item[(a)] The group $K$ is 
constructed from the Klein-Maskit combination theorems as an HNN-extension of the dihedral group 
$D_{n}=\langle A(z)=e^{2 \pi/n}z, B(z)=1/z\rangle$
by a cyclic group $\langle T \rangle$, where $T$ is a loxodromic transformation conjugating $B$ into $AB$. That is, $A=[T,B]=TBT^{-1}B$ and
$K=\langle B,T: B^{2}=[T,B]^{n}=([T,B]B)^{2}=1\rangle$.

\item[(b)] If $n = 2$, then $K$ is either constructed as above or it can be 
constructed from the Klein-Maskit combination theorems as an HNN-extension of the alternating group 
${\mathcal A}_{4}=\langle A(z)=i(1-z)/(z+1), B(z)=-z\rangle$
by a cyclic group $\langle T \rangle$, where $T$ is a loxodromic transformation commuting with $A$. That is,
$K=\langle A,B,T: A^{3}=B^{2}=(BA)^{3}=[T,A]=1\rangle$.

\end{enumerate}

\item Every Kleinian group constructed as in (1) is an origami-Schottky group of type $n$.

\end{enumerate}
\end{theo}

\begin{rema}\label{obs3}
The deck group of a regular origami pair cannot always be realized in terms of origami-Schottky groups. For instance, Example 7(2)  in \cite{SW} is a Hurwitz translation surface which is not of Schottky type as there is no surjective homomorphism from $K$ (as in Theorem \ref{const0} for $n=2$) to the dihedral group $D_{4}$ with a torsion-free kernel.
\end{rema}

\section{Preliminaries}
\subsection{Kleinian and Schottky groups}
A {\it Kleinian group} is a discrete subgroup $K$ of the group of M\"obius transformations ${\mathbb M}={\rm PSL}_{2}({\mathbb C})$ and its region of discontinuity is the open subset $\Omega$ of $\widehat{\mathbb C}$ (it might be empty) of points on which $K$ acts discontinuously.  If $\Gamma$ is a finite index subgroup of $K$, then both have the same region of discontinuity (see \cite{Maskit:book} for this and more generalities on Kleinian groups).  A {\it function group} is a pair $(K,\Delta)$, where $K$ is a finitely generated Kleinian group and $\Delta \subset \Omega$ is a $K$-invariant component. 

A {\it Schottky group of rank $g \geq 1$} is a Kleinian group $\Gamma<{\mathbb M}$
generated by loxodromic transformations $A_1,\ldots,A_g$, where there are $2g$ disjoint
simple loops, $\alpha_1, \alpha'_1,\ldots,\alpha_g,\alpha'_g$, with a common outside $D$ in the
extended complex plane $\widehat{\mathbb C}$, where $A_i(\alpha_i)=\alpha'_i$, and $A_i(D)\cap
D=\emptyset$, $i=1,\ldots,g$.  It is well known that (i) $\Omega$ is connected (so $(\Gamma,\Omega)$ is a function group) and dense in $\widehat{\mathbb C}$,
and that (ii) $S=\Omega/\Gamma$ is a closed Riemann surface of genus $g$. We say that $S$ is {\it uniformized} by $\Gamma$.
 As a consequence of the retrosection theorem, see for instance \cite{Bers, Koebe}, every closed Riemann surface can be uniformized by a Schottky group.  
It is well known that a Schottky group of rank $g$ is isomorphic to the free group of rank $g$ and it is purely loxodromic (moreover,  these properties characterize these groups within the class of Kleinian groups with non-empty region of discontinuity \cite{Maskit:Schottky groups}). 

\subsection{Maskit's description of function groups}
Maskit's decomposition theorem \cite{Maskit:construction, Maskit:function2, Maskit:function3, Maskit:function4} states that every function group can be constructed from elementary groups (Kleinian groups with finite limit set), quasifuchsian groups  (function groups with limit set being a simple loop) and degenerate groups (function groups whose region of discontinuity is connected and simply-connected) by a finite number of applications of the Klein-Maskit combination theorems \cite{Maskit:Comb, Maskit:Comb4, Maskit:book}. Moreover, in the construction, the amalgamated free products and the HNN-extensions are produced along cyclic groups. As a consequence of such a decomposition result, a Schottky group can be defined as a purely loxodromic geometrically finite function group with a totally disconnected limit set. The groups $K$ as in Theorem \ref{const0} are examples of geometrically finite function groups with totally disconnected limit sets. So, if $\Gamma$ is any finite index subgroup of $K$, which is torsion-free, then it is a Schottky group.

\subsection{Equivariant loops}
To prove our main result, we will need the following result which provides the existence of an invariant collection of loops on a surface. This result, in the case of compact $3$-manifolds, is a consequence of the Equivariant Loop Theorem \cite{M-Y}, whose proof is based on minimal surfaces, that is, surfaces that locally minimize the area. 

\begin{theo}\cite{Hidalgo:Auto}\label{lifting}
Let $(\Gamma,\Delta)$ be a torsion-free function group uniformizing a closed Riemann surface $S$ of genus $g \geq 2$, that is, there is a regular covering $P:\Delta \to S$ with $\Gamma$ as covering group. If $G$ is a group of automorphism of $S$,  then it lifts with respect to $P$ if and only if  there is a collection $\mathcal F$ of pairwise disjoint simple loops on $S$ such that:
\begin{itemize}[leftmargin=*,align=left]
\item[(i)] $\mathcal F$ defines the regular planar covering $P:\Delta \to S$; and
\item[(ii)] $\mathcal F$ is invariant under the action of $G$.
\end{itemize}
\end{theo}

\begin{rema}\label{observa1}
(1) In the case that $\Gamma$ is a Schottky group (so $\Delta$ is all of its region of discontinuity), $S \setminus {\mathcal F}$ consists of planar surfaces. (2) By the uniformization theorem, the Riemann surface $S$ has a natural hyperbolic structure, induced by the hyperbolic structure of the hyperbolic plane,  so that $G$ acts as a group of isometries. As every homotopically non-trivial simple loop on $S$ is homotopic to a unique simple closed geodesic, the collection ${\mathcal F}$ can be assumed to be formed by simple closed geodesics. 
\end{rema}

\section{Proof of Theorem \ref{const0}}\label{inter}
\subsection{Proof of Part (1)}
Let $K$ be an origami-Schottky group of type $n \geq 2$ with region of discontinuity $\Omega$. By the definition, $E=\Omega/K$ is an orbifold of genus one with exactly one cone point, of order $n$, and
there is a Schottky group $\Gamma$ which is a finite index subgroup of $K$. We may assume $\Gamma$ to be a normal subgroup; in which case we set $G=K/\Gamma$. The finite index property asserts that $\Omega$ is also the region of discontinuity of $\Gamma$. 
Let us consider a regular branched covering map $\pi:\Omega \to E$ (with deck group $K$), a regular covering $P:\Omega \to S=\Omega/\Gamma$ (with deck group $\Gamma$), and a regular branched cover $Q:S \to E$ (with deck group $G$) so that $\pi=Q \circ P$ (note that $G$ cannot be an abelian group as $Q$ has exactly one branch value of positive order).
There is a surjective homomorphism $\Phi:K\to G$, with kernel $\Gamma$, satisfying that $P \circ k = \Phi(k) \circ P$, for every $k \in K$.

As $\Gamma$ has no parabolic elements and $K$ is a finite extension of $\Gamma$, neither $K$ has parabolic elements (i.e.,  every element of $K$, different from the identity, is either loxodromic or elliptic). As $\Gamma$ is geometrically finite and of finite index, $K$ is also a geometrically finite function group.  As a consequence of the results in \cite{Hidalgo:MEFP} we have the following fact.

\begin{lemm} 
Let $k \in K$ be an elliptic transformation different from the identity. Then 
either (i) both of its fixed points belong to $\Omega$ or (ii) there is a loxodromic transformation in $K$ commuting with it.
\end{lemm}

As the group $G$ lifts under $P$, and such lifting is the group $K$, Theorem \ref{lifting} asserts the existence of a collection of pairwise disjoint simple loops $\mathcal F$ in $S$, invariant under $G$, such that (i) it divides $S$ into planar surfaces (as $\Gamma$ is a Schottky group) and (ii) each loop lifts to a simple loop on $\Omega$.  As observed in (2) of Remark \ref{observa1}, one may assume each of the loops in ${\mathcal F}$ to be a geodesic.
We will assume that the choice of ${\mathcal F}$ is minimal in the sense that if we delete some of the loops of ${\mathcal F}$, then either (a) we lose the $G$-invariance property or (b) one of the complements of these new set of loops is no longer planar.
Let  ${\mathcal G} \subset \Omega$ be the collection of loops, called the {\it structural loops},  obtained by the lifting of those in $\mathcal F$ under the covering $P$. The connected components of $\Omega\setminus{\mathcal G}$ are called the {\it structure regions}. By the planarity property of $\Omega$, each structure loop is a common boundary of exactly two structure regions.
As $P:\Omega \to S$ is a regular covering, whose deck group $\Gamma$ has no torsion, we may observe the two following facts:

\begin{enumerate}[leftmargin=*,align=left]
\item If $\delta \in {\mathcal G}$ and $\gamma \in {\mathcal F}$ so that $P(\delta)=\gamma$, then $P:\delta \to \gamma$ is a homeomorphism.
\item If $R$ be a structure region, then  $P:R \to P(R)$ is a homeomorphism.
\end{enumerate}

\begin{lemm}\label{lema1}
If $\delta \in {\mathcal G}$, then its $K$-stabilizer is either trivial or a cyclic group generated by an elliptic transformation keeping invariant each of the two structure regions sharing $\delta$ in the boundary.
\end{lemm}
\begin{proof}
Let $\delta \in {\mathcal G}$  and $P(\delta)=\gamma \in {\mathcal F}$. As $P:\delta \to \gamma$ is a homeomorphism, if $K_{\delta}$ denotes the $K$-stabilizer of $\delta$ and $G_{\gamma}$ denotes the $G$-stabilizer of $\gamma$, then $\Phi:K_{\delta} \to G_{\gamma}$ is an isomorphism. In particular, $K_{\delta}$ is a finite group of M\"obius transformations. The non-trivial finite groups of M\"obius transformations are cyclic groups, dihedral groups, the alternating groups ${\mathcal A}_{4}$ and ${\mathcal A}_{5}$ and the symmetric group ${\mathfrak S}_{4}$. Since neither the alternating groups nor the symmetric group can leave invariant a simple loop on the Riemann sphere, we obtain that the only possibilities for $K_{\delta}$ are, apart from the trivial group, a cyclic or dihedral group.
Let $R_{1}$ and $R_{2}$ the two structures regions sharing $\delta$ in their borders.
If $k \in K_{\delta}$ is different from the identity, then the only possibilities are for $k$ to be either (i) an elliptic element of order two permuting $R_{1}$ with $R_{2}$ or (ii) an elliptic element keeping invariant each $R_{j}$. 
In case (i), the group $K_{\delta}$ is either a cyclic group of order two (both of its fixed points belonging to $\delta$) or a dihedral group generated by an elliptic element of order two with both fixed points in $\delta$ and an elliptic element keeping invariant each $R_{j}$. This will assert that $E$ must have at least two cone points of order two, a contradiction.
As a consequence, the stabilizer group $K_{\delta}$ is either trivial or a cyclic group generated by an elliptic element keeping invariant each of structure regions $R_{1}$ and $R_{2}$.
\end{proof}

As a consequence of Lemma \ref{lema1}, for each $\gamma \in {\mathcal F}$ and each $\delta \in {\mathcal G}$ so that $P(\delta)=\gamma$, one has that $\pi(\delta)=Q(\gamma)$ is a power of a simple loop which does not contains the cone point of $E=\Omega/K$.

\begin{lemm}\label{lema3}
(1) All the loops in ${\mathcal F}$ are $G$-equivalent. Equivalently, all structure loops are $K$-equivalent. (2)
 All structure regions are $K$-equivalent.
\end{lemm}
\begin{proof}
(1) In fact, if this is not the case, then there exist two different loops $\gamma_{1}, \gamma_{2} \in {\mathcal F}$ so that $Q(\gamma_{1})$ and $Q(\gamma_{2})$ bounds a cylinder ${\mathcal U} \subset E$. As these two are homotopic in $E$ minus the cone point, $G_{\gamma_{1}}$ and $G_{\gamma_{2}}$ are both trivial or both isomorphic to a cyclic group of same order $m$. It follows that any component of $Q^{-1}({\mathcal U})$ must be a cylinder, a contradiction to the fact that no two elements of ${\mathcal F}$ can be homotopic (recall that these loops are simple closed geodesics).
(2) Part (1) above asserts that all loops in ${\mathcal F}$ are projected under $Q$ to the same loop $\alpha$ in $E$, and such a loop does not contain the cone point of $E$. As $E\setminus\{\alpha\}$ is connected, if $R_{1}$ and $R_{2}$ are any two structure regions, then they should be projected under $\pi$ to $E\setminus\{\alpha\}$.

\end{proof}

Let us fix one of the structure regions, say $R$, and let us denote by $\delta_{1},\ldots, \delta_{r}$ those structure loops in the boundary of $R$.
Let $K_{R}$ be the $K$-stabilizer of $R$ and by $K_{\delta_{j}}$ the $K$-stabilizer of $\delta_{j}$, which is either trivial or a cyclic group by Lemma \ref{lema1}. As all the structure loops are $K$-equivalent, by part (1) of Lemma \ref{lema3}, all $K$-stabilizers $K_{\delta_{j}}$ are conjugated in $K$. As in the above, we denote by $\alpha \subset E$ the loop obtained by projecting under $Q$ the loops if ${\mathcal F}$.

\begin{lemm}\label{lema5}
The $K$-stabilizer of any structure loop is a non-trivial cyclic group. 
\end{lemm}
\begin{proof}
We already know that the $K$-stabilizer of any structure loop is either a non-trivial cyclic group or the trivial group.
As a consequence of part (1) of Lemma \ref{lema3}, any two structure loops have conjugated $K$-stabilizers. So, if one has a trivial stabilizer, then everyone does. Assume the $K$-stabilizer of the structure loops to be trivial. It follows that the $G$-stabilizer of any of the loops in ${\mathcal F}$ is trivial. If $F$ is any of the connected components of $Q^{-1}(E\setminus\{\alpha\})$, then it is the complement of some closed discs of some closed surface $\widehat{F}$.
Let us consider the restriction $Q:F \to E\setminus\{\alpha\}$. We may glue discs to the boundary loops of the cylinder $E\setminus\{\alpha\}$ and and those of $F$ to extend continuously $Q:\widehat{F} \to \widehat{\mathbb C}$ as a covering map from a closed surface to the sphere with exactly one cone point of order $n \geq 2$. This is not possible by the Euler characteristic.
\end{proof}

As a consequence of Lemma \ref{lema5}, $K_{\delta_{j}}=\langle k_{j} \rangle$, where $k_{j} \in K$ is non-trivial elliptic element. As $k_{j} \in K_{R}$, we also have that $K_{R}$ is non-trivial.

\begin{lemm}\label{lema6}
The group $K_{R}$ cannot be cyclic.
\end{lemm}
\begin{proof}
Let us assume that $K_{R}$ is a cyclic group, say generated by the elliptic element $V \in K$, $V \neq I$. As $R/K_{R}$ only has one cone point, this of order $n$, it follows that $V$ has order $n$ and only one of its fixed points of $V$ belongs to $R$. Then there is a unique structure loop $\delta$ on the boundary of $R$ whose $K$-stabilizer is the cyclic group $K_{R}=\langle V \rangle$. Then all of the other structure loops on the boundary of $R$ must be stabilized by the identity in $K_{R}$. As all the other boundary structure loops of $R$ are $K$-equivalent to $\delta$, each of them must also have a non-trivial stabilizer; that being a cyclic group of the same order as $H$, it follows that $R$ can only have one boundary structure loop. This is not possible.
\end{proof}

\begin{lemm}\label{lema7}
The group $K_{R}$ is either a dihedral group $D_{n}$ or ${\mathcal A}_{4}$.
\end{lemm}
\begin{proof}
By Lemma \ref{lema6}, $K_{R}$ cannot be the trivial group nor a cyclic group. So the only possibilities we need to rule out are for it to be either ${\mathcal A}_{5}$ or ${\mathfrak S}_{4}$. If $K_{R}={\mathcal A}_{5}$, then (as $R/K_{R}$ only has one cone point) there should be two boundary structure loops $\delta_{1}$ and $\delta_{2}$, each one invariant under an element of order either $2$, $3$ or $5$, but of different orders. This is a contradiction to the fact that all structure loops have isomorphic $K$-stabilizers. The argument for $K_{R}={\mathfrak S}_{4}$ is similar.
\end{proof}

Using the same arguments at the end of the previous proof, for each of the possible groups $D_{n}$ and ${\mathcal A}_{4}$, one may obtain the following facts.

\noindent
(1) If $K_{R}=D_{n}=\langle a, b: a^{n}=b^{2}=(ab)^{2}=1\rangle$, then we may assume that the boundary structure loops are those whose stabilizers are the elements conjugated to $b$ and $ab$. The two fixed points of $a$ will project to the cone point of $E$. In this case, we have $r=2n$. Up to conjugation by a suitable M\"obius transformation, we may assume that $a=A$ and $b=B$ as in the theorem.

\noindent
(2) If $K_{R}={\mathcal A}_{4}=\langle a, b: a^{3}=b^{2}=(ab)^{3}=1\rangle$, then we may assume that the boundary structure loops are those whose stabilizers are the elements conjugated to $a$ and $ab$. The two fixed points of $b$ will project to the cone point of $E$. In this case $r=8$. Up to conjugation by a suitable M\"obius transformation, we may assume that $a=A$ and $b=B$ as in the theorem.

\medskip

Now, for each $j \in \{1,\ldots,r\}$ there is some $j_{i} \in \{1,\ldots, r\}$ and some $T_{j} \in K \setminus K_{R}$ so that $T_{j}(\delta_{j})=\delta_{j_{i}}$ and $T_{j}(R) \cap R = \emptyset$. We claim that $j_{i} \neq i$. In fact, if $j_{i}=j$, then $T_{j}$ will permute $R$ with the other structure region bounded by $\delta_{j}$; which is not possible by Lemma \ref{lema1}. In particular, each $T_{j}$ is a loxodromic transformation. 
If there exists some $k \in K_{R}$, $k \neq I$, such that $k(\delta_{j})=\delta_{j_{i}}$, then $k^{-1}T_{j} \in K$ will be an elliptic elements of order two keeping invariant $\delta_{j}$ and permuting $R$ with the other structural region having $\delta_{j}$ as boundary loop. This situation is not possible by Lemma \ref{lema1}.

In case (1),  the above asserts that there is a loxodromic element $T \in K$ so that $T$ sends the boundary loop around one of the fixed points of $b$ to the boundary loop around the fixed points of $ab$ which is not in the orbit by $a$. See Figure \ref{figura1} (for $n=3$ and $n=4$ ).

In case (2), as in the previous case, there is a loxodromic element $T \in K$ so that $T$ sends the boundary loop around one of the fixed points of $a$ to the boundary loop around the other fixed point of $a$. See Figure \ref{figura1}.
 All the above provides the desired result.

\begin{figure}[h]
\begin{center}
\includegraphics[width=3.5cm,keepaspectratio=true]{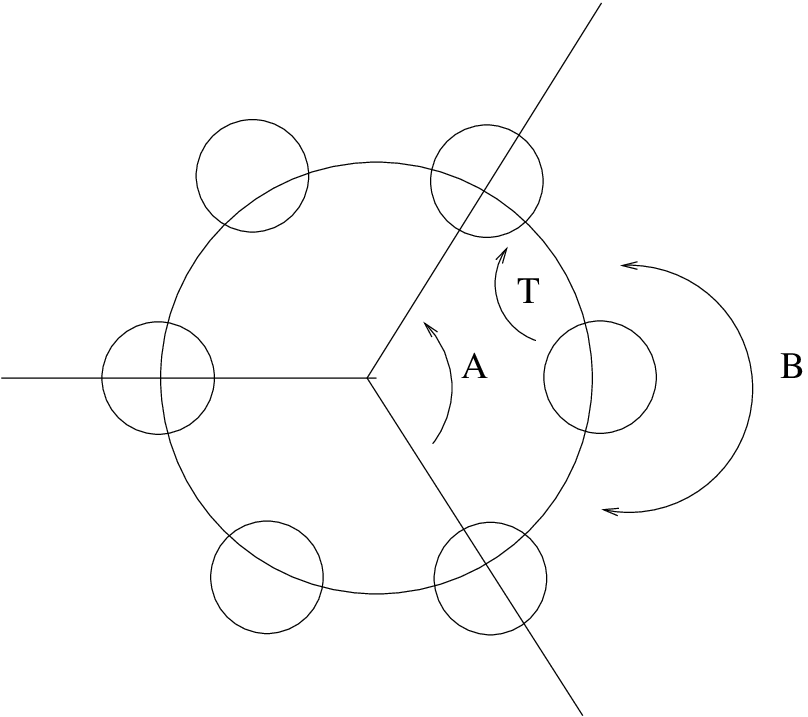} \;
\includegraphics[width=3.5cm,keepaspectratio=true]{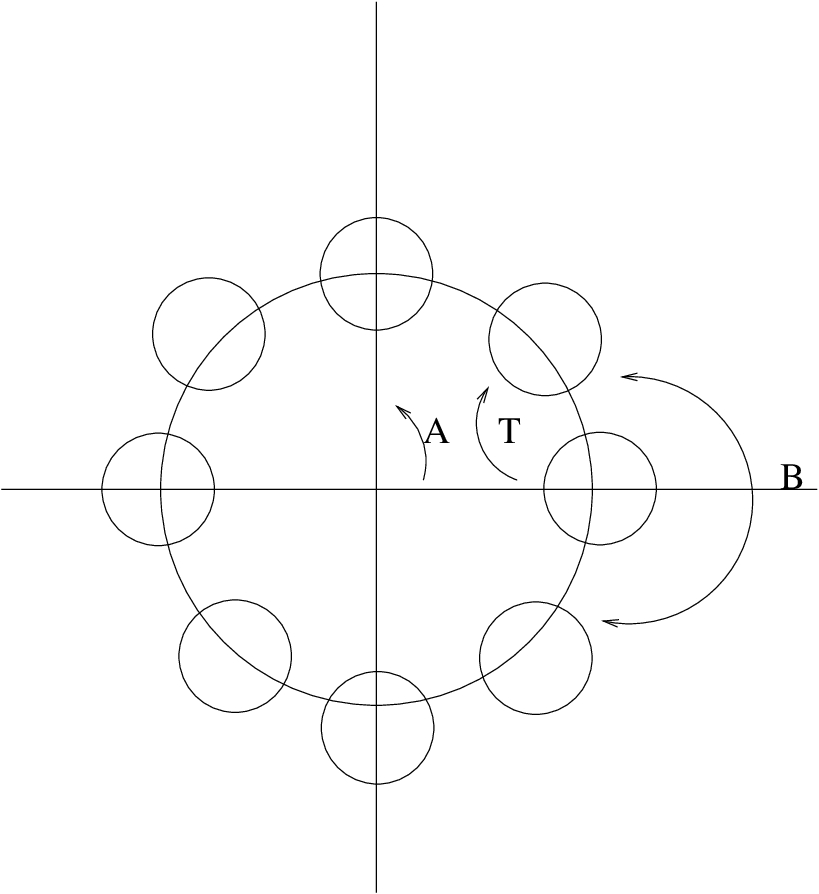} \; 
\includegraphics[width=3.5cm,keepaspectratio=true]{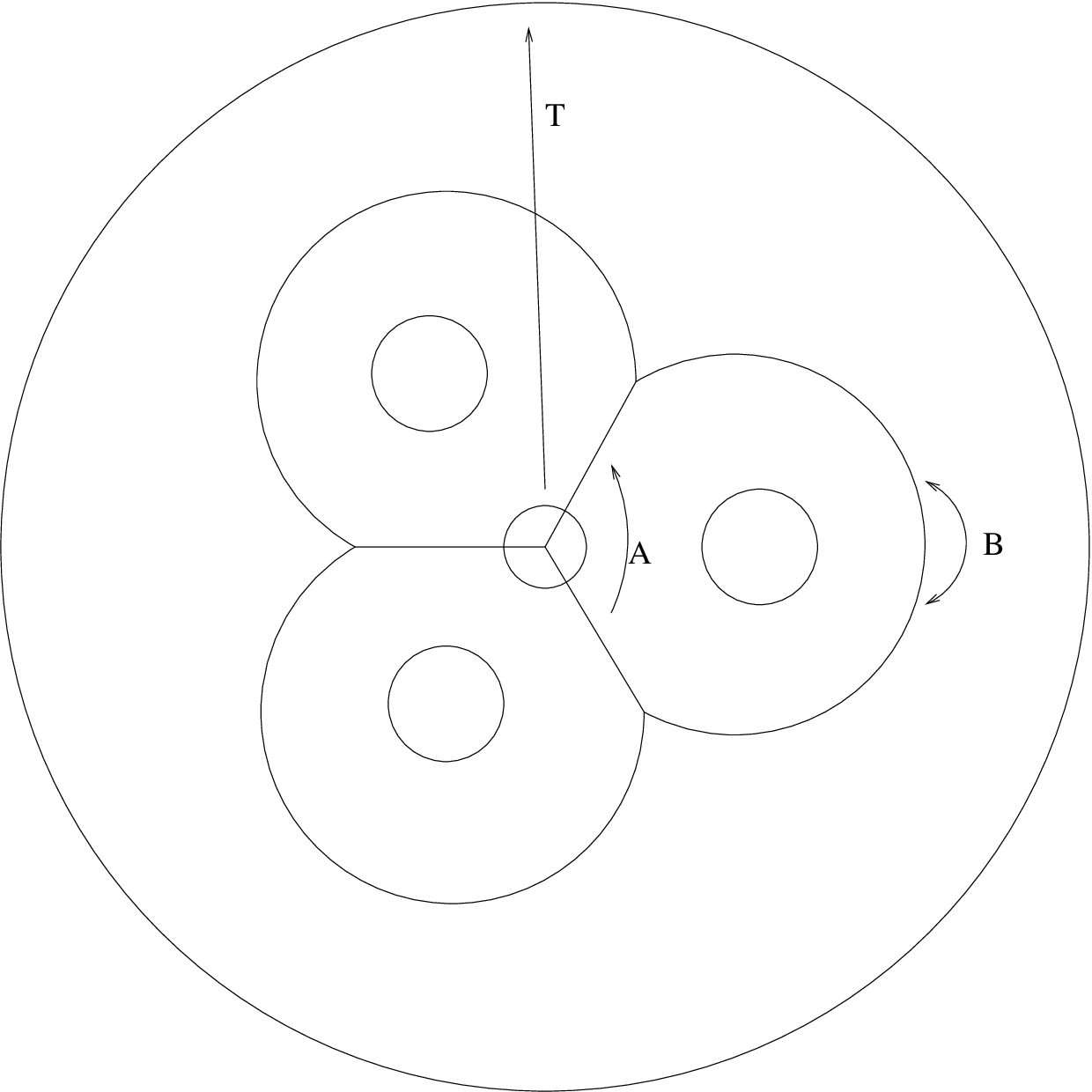}
\caption{$D_{3}$, $D_{4}$ and ${\mathcal A}_{4}$}
\label{figura1}
\end{center}
\end{figure}

\subsection{Proof of Part (2)}
Let $K$ be a Kleinian group which is constructed as in part (1) of the theorem. To see that $K$ is an origami-Schottky group, we need to find a finite index Schottky subgroup $\Gamma$ of $K$. 

\subsubsection{Construction in (\emph{a})}
If $n \geq 3$ odd,  then the subgroup
$$\Gamma=\langle C,  ACA^{-1}, A^{2} C A^{-2},\ldots, A^{-2} C A^{2}, A^{-1} C A \rangle,$$ where
$C=A^{(n-1)/2}T$, is a Schottky group of rank $n$, which a normal subgroup of index $2n$. This provides a regular origami of genus $n$ with automorphism group $G=D_{n}$.

If $n \geq 2$ even, then  the subgroup
$$\Gamma=\langle T,  ATA^{-1}, A^{2} T A^{-2},\ldots, A^{-2} T A^{2}, A^{-1} T A \rangle$$ is a Schottky group of rank $n$, which is not a normal subgroup, of index $2n$.  It is possible to find normal subgroups $N$ of finite index in $K$ being Schottky groups of rank $g>2$ and also being a subgroup of $\Gamma$ (for instance, $N=\cap_{k \in K} k \Gamma k^{-1}$). For $n=2$, these subgroups $N$ will provide examples of regular origamis $\eta: S=\Omega/N \to \Omega/K=E$ of genus $g>2$ with deck group $G=K/N$ of maximal order $4(g-1)$ (i.e. Hurwitz translation surfaces) admitting a non-normal subgroup $H$ so that $S/H$ is the genus two surface $\Omega/\Gamma$. 

\subsubsection{Construction in (\emph{b})}
The subgroup 
$$\Gamma=\langle T,  BTB, AB T BA^{-1},A^{-1}B T B A \rangle$$ is a Schottky group of rank $4$, which a normal subgroup of index $12$. This subgroup $\Gamma$  provides a regular origami of genus $g=4$ with automorphism group $G={\mathcal A}_{4}$ of maximal order $4(g-1)=12$ (see example 7(3) in page 9 of \cite{SW}). More generally, if $N$ is a torsion-free normal subgroup of index $d$ of $K$, then $N$ is a Schottky group of rank $g=1+d/4$ (so $d$ is necessarily divisible by $4$) defining a regular origami of maximal order $4(g-1)$ (more examples of Hurwitz translation surfaces). Examples of such $N$ are of the form $\Gamma^{(r)}[\Gamma:\Gamma]$, for any positive integer $r$ (in this case, $\Gamma/N \cong {\mathbb Z}_{r}^{n}$).

\section{Some Final Remarks}
\subsection{A relation to $p$-adic origamis}
There is a version of origamis at the level of $p$-adic fields, restricted to the case of Mumford curves. These correspond to $p$-adic Kleinian groups $K$ being finite extensions of a $p$-adic Schottky group $\Gamma$, such that $K$ unifomizes a Mumford curve of genus one with exactly one cone point \cite{Herrlich,Kremer} (see \cite{GvdP} for generalities on $p$-adic Schottky groups). In \cite{Kremer}, Kremer describes those $p$-adic regular origamis. In there it is proved that for $p>5$, the quotient group $K/\Gamma$ is either the dihedral group  $D_{n}$, for $n \geq 3$, or ${\mathcal A}_{4}$. The results of the present paper provide the corresponding counterpart at the level of origamis in the complex numbers and do not restrict to regular ones. The proof of Theorem \ref{const0} uses the existence of structural loops and structural regions which in some sense correspond, in the work in \cite{Kremer}, to a certain sub-tree of the Bruhat-Tits tree.

\subsection{A relation to $3$-manifolds}
Let $K$ be an origami-Schottky group, with a region of discontinuity $\Omega$,  and let $\Gamma$ be a finite index Schottky subgroup. If ${\mathbb H}^{3}$ denotes the hyperbolic $3$-space and ${\bf H}:={\mathbb H}^{3} \cup \Omega$, then ${\mathcal O}:={\bf H}/K$ is a $3$-dimensional compact orbifold (topologically equivalent to a genus one handleblody) whose conical set consists of an interior simple loop (with branch order $k$) union an arc (with branch order $n \geq 2$) connecting it to a point in the boundary and $(k,n) \in \{(2,n),(3,2)\}$. Similarly, $M={\bf H}/\Gamma$ is a handlebody of genus $g \geq 2$. It is well known that every compact $3$-manifold (without boundary) is a union of two handlebodies of the same genus glued along their boundaries (Heegaard splitting). By gluing to copies $M_{1}$ and $M_{2}$ of the above handlebodies (with the same $(k,n)$) one obtains a compact $3$-manifold $M_{1,2}$ without boundary and, similarly, by gluing along the boundaries the corresponding orbifolds ${\mathcal O}_{1}$ and ${\mathcal O}_{2}$ one obgains a compact $3$-orbifold ${\mathcal O}_{1,2}$ without boundary (the conical set is formed by a link of two loops, with the same cone order $k$, union a bridge connecting them of order $n$). This can be thought of as a $3$-dimensional version of an origami.


\end{document}